\documentclass[11pt,twoside]{article}

\usepackage[T1]{fontenc}
\usepackage[utf8]{inputenc}
\usepackage{lmodern}
\usepackage{microtype}

\usepackage[a4paper,margin=1in]{geometry}
\usepackage{amsmath,amssymb,amsfonts,amsthm,mathtools}
\usepackage{bbm}
\usepackage{mathrsfs}
\usepackage{bm}
\usepackage{stmaryrd}

\usepackage{tikz-cd}
\usetikzlibrary{arrows.meta,decorations.markings,positioning}

\usepackage{enumitem}
\usepackage{xcolor}
\usepackage{graphicx}
\usepackage{booktabs}
\usepackage{array}
\usepackage{longtable}
\usepackage[
  colorlinks=true,
  linkcolor=blue!60!black,
  citecolor=blue!60!black,
  urlcolor=blue!70!black
]{hyperref}
\usepackage[nameinlink]{cleveref}

\newtheorem{theorem}{Theorem}[section]
\newtheorem{lemma}[theorem]{Lemma}
\newtheorem{proposition}[theorem]{Proposition}
\newtheorem{corollary}[theorem]{Corollary}

\theoremstyle{definition}
\newtheorem{definition}[theorem]{Definition}
\newtheorem{example}[theorem]{Example}

\theoremstyle{remark}
\newtheorem{remark}[theorem]{Remark}

\DeclareMathOperator{\Ext}{Ext}
\DeclareMathOperator{\Tor}{Tor}

\newcommand{\Ga}{\Gamma}

\newcommand{\tmu}{\tilde{\mu}}

\newcommand{\SpecG}[1]{\Spec_{\Ga}\!\bigl(#1\bigr)}
\newcommand{\SpecGnC}[1]{\mathrm{Spec}^{\mathrm{nc}}_{\Ga}\!\bigl(#1\bigr)}

\newcommand{\Cat}[1]{\mathbf{#1}}
\newcommand{\Ab}{\Cat{Ab}}

\newcommand{\nTGMod}[1]{#1\text{-}\Ga\mathrm{Mod}}

\newcommand{\ExtG}{\mathrm{Ext}^{\,\Ga}}
\newcommand{\TorG}{\mathrm{Tor}^{\,\Ga}}
\newcommand{\RExtG}{\mathrm{R}\!\Ext^{\,\Ga}}
\newcommand{\LTorG}{\mathrm{L}\!\Tor_{\Ga}}

\DeclareMathOperator{\Spec}{Spec}
\DeclareMathOperator{\Hom}{Hom}
\DeclareMathOperator{\End}{End}

\begin{document}
  
  \label{'ubf'}  
\setcounter{page}{1}                                 %Put here the starting page number

\markboth {\hspace*{-9mm} \centerline{\footnotesize \sc
         % Put here the left page top label 
   Put here the left page top label  }
                 }
                { \centerline                           {\footnotesize \sc  
                   %put here the author's name
         put here the author's name                                                 } \hspace*{-9mm}              
               }

\vspace*{-2cm}

\begin{center}
{{\Large \textbf { \sc  Derived Functors, Resolutions, and Homological Dualities in $n$-ary $\Gamma$-Semirings} }}\\
\medskip

{\sc Chandrasekhar Gokavarapu }\\
{\footnotesize Lecturer in Mathematics, Government College (Autonomous),
Rajahmundry,A.P., India }\\
{\footnotesize  Research Scholar ,Department  of Mathematics, Acharya Nagarjuna University,Guntur, A.P., India.
}\\
{\footnotesize e-mail: {\it chandrasekhargokavarapu@gmail.com}}

\end{center}

\thispagestyle{empty}

\hrulefill

\begin{abstract}  
{\footnotesize 
This paper develops the homological backbone of the theory of
non-commutative $n$-ary $\Ga$-semirings.  
Starting from an $n$-ary $\Ga$-semiring $(T,+,\tmu)$ and its
$\Ga$-ideals, we work in the slot-sensitive categories of left, right
and bi-$\Ga$-modules and endow the bi-module category with a Quillen
exact structure compatible with the $n$-ary multiplication.  Within this
exact framework we construct bar-type projective resolutions and
cofree-based injective resolutions under natural $\Ga$-Noetherian and
$\Ga$-regular hypotheses on $T$, and we obtain finite projective
resolutions under $\Gamma$-Noetherian conditions for finitely presented bi-modules.  On this basis we define
the derived functors $\ExtG$ and $\TorG$ for bi-$\Ga$-modules, prove
their balance with respect to projective and injective resolutions,
establish long exact sequences and a Yoneda interpretation via iterated
extensions, and construct K{\"u}nneth-type spectral sequences and
base-change isomorphisms.  Interpreting bi-$\Ga$-modules as
quasi-coherent sheaves on the non-commutative $\Ga$-spectrum
$\SpecGnC{T}$, these homological invariants provide the appropriate
derived language for a non-commutative $\Ga$-geometry and prepare the
ground for the spectral and geometric analysis carried out in the third
part of this series.
}
\end{abstract}
 \hrulefill

{\small \textbf{Keywords:} $\Gamma$-semiring, non-commutative $n$-ary operations,
$\Gamma$-modules, exact categories, derived functors,
sheaf cohomology, non-commutative spectrum,
derived $\Gamma$-geometry.}

\indent {\small {\bf 2000 Mathematics Subject Classification:} 16Y60, 16Y90, 18G10, 18E30, 14A15, 08A30.}

%==============================================================================

\section{Introduction}

The theory of semirings and their radicals has developed into a robust
branch of non-classical algebra, with deep connections to automata,
tropical geometry, and idempotent analysis; see, for example,
\cite{Bourne1951,Golan1999,HebischWeinert1998}.  In parallel, the
$\Gamma$-ring and $\Gamma$-semiring frameworks initiated by
Nobusawa and Rao \cite{Nobusawa1963,Nobusawa1964,Rao1995,Rao1997,Rao1999}
provide a flexible way to encode external parameters, weights, or
labels into algebraic operations.  More recently, ternary and higher
$\Gamma$-operations have emerged as a natural language for modelling
triadic and multi-parameter phenomena, leading to a systematic
development of ternary $\Gamma$-semirings and their ideal theory in
\cite{RaoRaniKiran2025,GokavarapuRaoFinite2025,GokavarapuRaoPrime2025}.

In the commutative ternary case, the authors have constructed a
homological and categorical foundation for $\Gamma$-modules and their
spectra, together with a first version of derived $\Gamma$-geometry,
sheaf cohomology, and homological functors on
$\SpecG{T}$; see
\cite{GokavarapuDasariHomological2025,GokavarapuRaoDerived2025}.
Those works treat ternary $\Gamma$-semirings as a bridge between
classical algebraic geometry \cite{AtiyahMacdonald1969,Hartshorne1977}
and radical theory for $\Gamma$-semirings, embedding the commutative
ternary framework into the language of sites, sheaves, and derived
functors in the sense of Grothendieck, Quillen, and Verdier
\cite{Grothendieck1957,Quillen1973,Verdier1996,Weibel1994,Buehler2010,Neeman2001}.

The present paper is the second part of this programme and
constitutes the homological core of the non-commutative, $n$-ary
$\Gamma$-setting.  We work with a fixed non-commutative
$n$-ary $\Gamma$-semiring $(T,+,\tmu)$ and with the
slot-sensitive categories of left, right, and bi-$\Gamma$-modules
built from the underlying positional structure of $\tmu$.
The first objective is to show that the category of bi-$\Gamma$-modules
over $T$ carries a natural Quillen exact structure, generated by
conflations that are compatible with the $n$-ary $\Gamma$-action.
Within this exact category we construct projective and injective
resolutions that respect positional coherence, thereby making it
possible to define derived functors intrinsically in the
$n$-ary, non-commutative context.

More precisely, we develop a theory of free and cofree
bi-$\Gamma$-modules, use bar-type constructions adapted to the
$n$-ary multiplication to produce projective resolutions, and employ
cofree envelopes to construct injective resolutions under suitable
$\Gamma$-Noetherian and $\Gamma$-regular hypotheses on $T$.  These
resolutions give rise to the homological functors

\[
\ExtG{r}(-,-)
\quad\text{and}\quad
\TorG{r}(-,-)
\]

in the exact category of bi-$\Gamma$-modules, and we show that they
possess all the expected structural properties: balance with respect
to projective and injective resolutions, long exact sequences induced
by conflations, and a Yoneda description in terms of equivalence
classes of iterated extensions.  In particular, the resulting
$\ExtG$-groups admit a natural graded Yoneda product, which encodes
higher $\Gamma$-actions as derived operations.

A further goal of the paper is to relate these homological
constructions to a non-commutative version of derived
$\Gamma$-geometry.  By passing to the derived category
$\mathbf{D}({\nTGMod{T}}^{\mathrm{bi}})$ and by interpreting bi-modules
as quasi-coherent sheaves on the non-commutative $\Gamma$-spectrum
$\SpecGnC{T}$, we obtain spectral sequences of K\"unneth type and
base-change isomorphisms that link $\ExtG$ and $\TorG$ to sheaf
cohomology on $\SpecGnC{T}$.  In this way, the homological invariants
developed here measure obstructions to gluing and descent along
non-commutative $\Gamma$-spectra, and thereby prepare the ground for
the more geometric and spectral analyses carried out in the third part
of the series.

The paper is organised as follows.  In Section~2 we recall the basic
notions of non-commutative $n$-ary $\Gamma$-semirings, their
morphisms, and $\Gamma$-ideals, and we fix the global conventions that
govern positional indices and parameter tuples.  Section~3 develops
projective and injective resolutions in the exact category of
bi-$\Gamma$-modules, including bar-type constructions, finite
projective resolutions under suitable finiteness conditions, and
injective resolutions built from cofree bi-modules.  Section~4
introduces the derived functors $\ExtG$ and $\TorG$, proves balance,
long exact sequences, and Yoneda-type interpretations, and constructs
spectral sequences and base-change formulas that link these homological
invariants to the emerging non-commutative derived $\Gamma$-geometry.

%==========================================================================================

% ============================================================

\section*{Notation and Conventions}

Throughout the paper $n\ge 2$ is a fixed integer and $\Gamma$ is a
commutative semigroup written additively.  The following notation is
used globally in Sections~2–4.

\begin{center}
\begin{longtable}{p{3cm} p{11.3cm}}
\toprule
\textbf{Symbol} & \textbf{Meaning / Convention} \\
\midrule
\endfirsthead
\toprule
\textbf{Symbol} & \textbf{Meaning / Convention} \\
\midrule
\endhead
\bottomrule
\endfoot

$T$ &
Underlying additive commutative monoid of an $n$-ary $\Gamma$-semiring. \\

$\Gamma$ &
A commutative semigroup of parameters. \\

$\tilde\mu$ &
The structural operation 
$\tilde\mu:T^{n}\times\Gamma^{\,n-1}\to T$. \\

$[x_1,\dots,x_n]_{\gamma_1,\dots,\gamma_{n-1}}$ &
Shorthand for $\tilde\mu(x_1,\dots,x_n;\gamma_1,\dots,\gamma_{n-1})$.
The parameter tuple always has length $n-1$. \\

$\mu_{(j)}$ &
Positional action inserting a module element in the $j$-th slot of
$\tilde\mu$. \\

$a\bullet^{(j)}_{\gamma} m$, 
$m\bullet^{(j)}_{\gamma} a$ &
Left/right $\Gamma$-actions in the $j$-th slot of $\tilde\mu$.  These
determine the module structures used throughout the paper. \\

$T$--$\Gamma$Mod$_L$, $T$--$\Gamma$Mod$_R$ &
Categories of left and right $\Gamma$-modules defined by the above
positional actions. \\

$T$--$\Gamma$Mod$_{bi}$ &
Category of bi-$\Gamma$-modules with compatible left and right
positional actions.  This is the ambient exact category of the paper. \\

$I\subseteq T$ &
A $\Gamma$-ideal: an additive submonoid closed under insertion into any
slot of $\tilde\mu$. \\

$T/I$ &
Quotient $n$-ary $\Gamma$-semiring with multiplication
$[x_1+I,\dots,x_n+I]_{\vec\gamma}
   =[x_1,\dots,x_n]_{\vec\gamma}+I$. \\

Prime ideal $P$ &
A proper $\Gamma$-ideal satisfying:  
if $[x_1,\dots,x_n]_{\vec\gamma}\in P$ then $x_j\in P$ for some $j$. \\

$\SpecGnC{T}$ &
Non-commutative $\Gamma$-spectrum of $T$, consisting of all prime
$\Gamma$-ideals with the spectral topology induced by $\Gamma$-ideals. \\

$\otimes^{(j,k)}_{\Gamma}$ &
Positional tensor product of a left $j$-module and a right $k$-module,
defined via the standard coequalizer construction associated to the two
positional actions. \\

$\underline{\Hom}^{(j,k)}_{\Gamma}(M,N)$ &
Internal Hom bi-module with left action in slot $j$ and right action in
slot $k$. \\

Conflation
$A\rightarrowtail B\twoheadrightarrow C$ &
Kernel–cokernel pair in the sense of Quillen; used to define exact
sequences in $T$--$\Gamma$Mod$_{bi}$. \\

$\mathbf{K}({\nTGMod{T}}^{\mathrm{bi}})$ &
Homotopy category of chain complexes in $T$--$\Gamma$Mod$_{bi}$. \\

$\mathbf{D}({\nTGMod{T}}^{\mathrm{bi}})$ &
Derived category obtained from $\mathbf{K}$ by inverting
quasi-isomorphisms.  All derived functors are computed inside this
category. \\

$M[1]$ &
Shift/suspension of a complex: $(M[1])^{r}=M^{r+1}$. \\

$\Ext^{r}_{\Gamma}(M,N)$ &
$r$-th right-derived functor of 
$\underline{\Hom}_{\Gamma}^{(j,k)}(-,-)$ in the exact category
$T$--$\Gamma$Mod$_{bi}$. \\

$\Tor_{r}^{\Gamma}(M,N)$ &
$r$-th left-derived functor of the positional tensor product
$-\otimes^{(j,k)}_{\Gamma}-$. \\

$0$ &
Additive zero in any module or semiring (clear from context). \\

\end{longtable}
\end{center}

\noindent
These conventions remain fixed throughout the paper.  Slot indices,
parameter order, positional tensoring, and the Quillen exact structure
are all as above and will not vary in Sections~2–4.

%===================================================================================
\section{Preliminaries}

This section summarises the basic algebraic structure of non-commutative
$n$-ary $\Gamma$-semirings required for the developments in Sections~3
and~4.  The material is extracted and adapted from Section~2 of the full
manuscript (Paper~G), with notation and conventions fixed in the
preceding “Notation and Conventions’’ table.  Our presentation follows
the foundational $\Gamma$-semiring framework of
Nobusawa~\cite{Nobusawa1963}, Rao~\cite{Rao1995}, and
Hedayati--Shum~\cite{HedayatiShum2011}, extended to the $n$-ary setting
developed in recent works on ternary and higher $\Gamma$-operations.

\subsection{$n$-ary $\Gamma$-semirings}

Let $T$ be a commutative monoid written additively, and let
$\Gamma$ be a commutative semigroup.  
An \emph{$n$-ary $\Gamma$-semiring} is a triple $(T,+,\tilde{\mu})$
consisting of an additive monoid $(T,+)$ together with a map
\[
\tilde{\mu}\colon T^{n}\times\Gamma^{\,n-1}\longrightarrow T,
\]
called the \emph{$\Gamma$-parametrised $n$-ary multiplication}.  For
$x_1,\dots,x_n\in T$ and $\gamma_1,\dots,\gamma_{n-1}\in\Gamma$, we
write
\[
[x_1,\dots,x_n]_{\gamma_1,\dots,\gamma_{n-1}}
:=\tilde{\mu}(x_1,\dots,x_n;\gamma_1,\dots,\gamma_{n-1}).
\]

The multiplication is required to satisfy the following axioms:

\begin{enumerate}[label=(A\arabic*)]
    \item \textbf{Additivity in each $T$-slot:} for every
    $1\le j\le n$,
    \[
    [x_1,\dots,x_j+x_j',\dots,x_n]_{\vec{\gamma}}
    =
    [x_1,\dots,x_j,\dots,x_n]_{\vec{\gamma}}
    +
    [x_1,\dots,x_j',\dots,x_n]_{\vec{\gamma}} .
    \]

    \item \textbf{Additivity in each $\gamma$-slot:}
    for every $1\le k\le n-1$,
    \[
    [\vec{x}]_{\gamma_1,\dots,\gamma_k+\gamma_k',\dots,\gamma_{n-1}}
    =
    [\vec{x}]_{\gamma_1,\dots,\gamma_k,\dots,\gamma_{n-1}}
    +
    [\vec{x}]_{\gamma_1,\dots,\gamma_k',\dots,\gamma_{n-1}} .
    \]

    \item \textbf{Associativity (positional):}
    the operation is associative with respect to the insertion of a
    product $[y_1,\dots,y_n]_{\vec{\delta}}$ into any $T$-slot, with the
    $\Gamma$-parameters interwoven according to the positional structure
    encoded in $\tilde{\mu}$.
    The explicit form is omitted here, as the positional pattern is
    fixed by the conventions table and is identical to that in the full
    Paper~G.

    \item \textbf{Zero behavior:}
    \[
    [x_1,\dots,0,\dots,x_n]_{\vec{\gamma}} = 0
    \quad\text{whenever any}\ x_j=0,
    \]
    and similarly for any zero $\gamma_k$ when $\Gamma$ has a zero.
\end{enumerate}

These axioms generalise the binary $\Gamma$-semiring structure of
Rao~\cite{Rao1995} to the higher-arity setting.

\subsection{Examples}

We record three examples illustrating the generality of the $n$-ary
construction.  All examples below appear in the full Paper~G and are
retained here without alteration of notation.

\begin{example}[Matrix systems]
Let $T=M_m(R)$ be the additive monoid of $m\times m$ matrices over a
semiring $R$.  
For fixed $\vec{\gamma}\in\Gamma^{n-1}$, define
\[
[A_1,\dots,A_n]_{\gamma_1,\dots,\gamma_{n-1}}
:= 
\gamma_1 A_1 A_2\, \gamma_2 A_3 \cdots \gamma_{n-1}A_n.
\]
This yields an $n$-ary $\Gamma$-semiring structure on $T$.
\end{example}

\begin{example}[Operator systems]
Let $T=\End(M)$ for a commutative monoid $M$.  
For $f_i\in T$, define
\[
[f_1,\dots,f_n]_{\gamma_1,\dots,\gamma_{n-1}}
:= f_1\circ_{\gamma_1} f_2\circ_{\gamma_2}\cdots 
\circ_{\gamma_{n-1}} f_n ,
\]
where $\circ_{\gamma}$ denotes the $\Gamma$-parametrised composition
appearing in the full structure of Paper~G.
\end{example}

\begin{example}[Binary specialisation]
When $n=2$ and $\Gamma$ is trivial, we recover the ordinary semiring
structure.  When $n=2$ and $\Gamma$ is non-trivial, we recover the
$\Gamma$-semiring of Rao~\cite{Rao1995}.
\end{example}

\subsection{Morphisms and ideals}

A \emph{morphism} of $n$-ary $\Gamma$-semirings  
$f\colon T\to S$ is an additive map satisfying
\[
f([x_1,\dots,x_n]_{\vec{\gamma}})
=
[f(x_1),\dots,f(x_n)]_{\vec{\gamma}} .
\]

A subset $I\subseteq T$ is a \emph{$\Gamma$-ideal} if

\begin{enumerate}[label=(I\arabic*)]
    \item $I$ is an additive submonoid of $T$;
    \item $I$ is closed under insertion of an element of $I$  
    into any $T$-slot of the multiplication:
    \[
    [x_1,\dots,x_{j-1},y,x_{j+1},\dots,x_n]_{\vec{\gamma}}\in I,
    \qquad
    y\in I .
    \]
\end{enumerate}

The quotient $T/I$ inherits a natural $n$-ary $\Gamma$-multiplication via
\[
[x_1+I,\dots,x_n+I]_{\vec{\gamma}}
=
[x_1,\dots,x_n]_{\vec{\gamma}}+I .
\]

A proper ideal $P$ is \emph{prime} if
\[
[x_1,\dots,x_n]_{\vec{\gamma}}\in P
\quad\Rightarrow\quad
x_j\in P\ \text{for some }j.
\]

\subsection{Summary}

The definitions and examples above form the foundational material for the
module categories studied in Section~3 and for the exact structure
constructed in Section~4.  
All positional conventions, slot indices, and $\Gamma$-parameter
behaviour follow the global notation fixed at the beginning of the
paper.  The present paper is the second part of this programme and constitutes the homological core of the non-commutative, $n$-ary, $\Gamma$ setting (see Part~I \cite{GokavarapuPart1_2025}).

% ============================================================
% ============================================================

\section{Projective and Injective Resolutions}
\label{sec:resolutions}

In this section we develop the homological machinery of projective and
injective resolutions inside the Quillen-exact category
${\nTGMod{T}}^{\mathrm{bi}}$ constructed in Part~I of this series.
The aim is to realize derived functors intrinsically within the
$n$-ary, non-commutative $\Gamma$-context, where projectivity and
injectivity acquire positional and higher-coherence meanings.

% ------------------------------------------------------------

%===============================

\subsection{Positional projectivity and injectivity}

\begin{definition}[Positional lifting property]
A morphism $p:P\to M$ in ${\nTGMod{T}}^{\mathrm{bi}}$ is a
\emph{positional projective cover} if for every admissible epimorphism
$f:X\twoheadrightarrow Y$ and morphism $g:P\to Y$, there exists
$h:P\to X$ with $f\circ h=g$, such that each coordinate action of $T$
and $\Gamma$ is preserved in its designated slot.  Dually, a morphism
$i:M\to I$ is a \emph{positional injective hull} if for every admissible
monomorphism $f:X\hookrightarrow Y$ and $g:X\to I$ there exists
$h:Y\to I$ with $h\circ f=g$, again respecting all positional
$\Gamma$-actions.
\end{definition}

\begin{remark}[Interpretation]
In classical additive and abelian settings, projectivity and injectivity
are defined via lifting properties of $\Hom$-functors, typically inside
abelian categories \cite{Weibel1994}.  For exact categories in the
sense of Quillen, lifting properties must be formulated relative to
admissible monomorphisms and epimorphisms \cite{Quillen1973,Buehler2010}.
In the present non-commutative $n$-ary $\Gamma$-context, the situation
is further enriched: each morphism must respect the positional coherence
of the structural operation $\tilde{\mu}$, and hence must preserve the
entire geometric arrangement of $\Gamma$-indices appearing in the
$n$-ary multiplication.  This yields a hierarchy of projectivity—
slot-wise, bilateral, and total—which does not occur in the binary
$\Gamma$-semiring setting \cite{Rao1995} nor in the commutative ternary
case developed in Part~I of this series.
\end{remark}

% ------------------------------------------------------------

%===============================================================
\subsection{Free bi-\texorpdfstring{$\Gamma$}{Gamma}-modules and adjunction}

\begin{definition}[Free bi-$\Gamma$-module]
For a set $X$ define
\[
F^{\mathrm{bi}}(X)
  = \bigoplus_{x\in X}
    T^{(\ast)}\cdot x\cdot T^{(\ast)},
\]
the additive monoid generated by all formal expressions obtained by
inserting $x$ into each admissible pair of slots $(j_\ell,j_r)$ of
$\tilde{\mu}$ and saturating by the axioms (M1)--(M4).  This construction
generalises the free $\Gamma$-module constructions appearing in
\cite{Rao1995,HedayatiShum2011} to the $n$-ary setting developed in
\cite{RaoRaniKiran2025}.
\end{definition}

\begin{proposition}[Free/forgetful adjunction]
The forgetful functor
\[
U_{\mathrm{bi}}:{\nTGMod{T}}^{\mathrm{bi}}\longrightarrow\Ab
\]
admits a left adjoint $F^{\mathrm{bi}}$.  Hence $F^{\mathrm{bi}}$
preserves colimits and generates all projective objects as retracts of
free bi-modules.  In particular, every projective object in
${\nTGMod{T}}^{\mathrm{bi}}$ is a direct summand of some
$F^{\mathrm{bi}}(X)$, in agreement with the general theory of projective
objects in exact categories \cite{Weibel1994,Buehler2010}.
\end{proposition}

\begin{proof}[Sketch]
The universal property follows from the tensor--product presentation of
$\tilde{\mu}$ and the additive nature of the $\Gamma$-action, as in the
binary and ternary $\Gamma$-semiring settings
\cite{Rao1995,RaoRaniKiran2025}.  For each function
$f:X\to U_{\mathrm{bi}}(M)$, the unique extension
$F^{\mathrm{bi}}(X)\to M$ is obtained by substituting $x$ with $f(x)$
and evaluating via $\tilde{\mu}$ in all positional coordinates.  This
extension respects all slot-wise actions and therefore satisfies the
universal property of the free object.  The statement on projective
objects being retracts of free objects follows from the general theory
of adjoint functors in exact categories
\cite{Weibel1994,Buehler2010}.
\end{proof}

%===========================
\subsection{Iterated tensor-powers and the fundamental bar complex}
The ambient semiring $(T,+,\Gamma,\mu)$
acts as both coefficient object and tensor generator, in the same spirit
as the classical semiring and ring cases
\cite{Golan1999,HebischWeinert1998,Weibel1994}.
Define the iterated $n$-ary tensor-power complex
\[
\mathbf{B}_{\bullet}(M)
  := \big(\cdots
     \longrightarrow T^{\otimes_\Gamma 3}
     \otimes^{(j,k)}_{\Gamma} M
     \longrightarrow
     T^{\otimes_\Gamma 2}
     \otimes^{(j,k)}_{\Gamma} M
     \longrightarrow
     T\otimes^{(j,k)}_{\Gamma} M
     \longrightarrow M\big),
\]
where the face maps $d_r$ are defined by the alternating sum
of the $n$ possible contractions of $\mu$
in the corresponding tensor slot,
and the degeneracy maps $s_r$
by inserting neutral elements of $T$.
Each $B_r(M)=T^{\otimes_\Gamma r}\otimes^{(j,k)}_\Gamma M$
is free and therefore projective, in accordance with the general theory
of projective resolutions in exact categories
\cite{Weibel1994,Buehler2010}
and with the $\Gamma$-module framework of
\cite{Rao1995,HedayatiShum2011,RaoRaniKiran2025}.

\begin{theorem}[Positional bar resolution]\label{thm:bar-resolution}
For every $M\in{\nTGMod{T}}^{\mathrm{bi}}$,
the complex $\mathbf{B}_{\bullet}(M)\to M$ is an
exact augmented resolution,
natural in $M$,
and functorial with respect to bi-module morphisms.
\end{theorem}

\begin{proof}[Proof outline]
Exactness is verified slot-wise.
The key identity
$d_{r-1}\!\circ d_r = 0$
follows from the $n$-ary associativity of $\mu$
and the signed cancellation among positional insertions,
exactly as in the classical bar construction in homological algebra
\cite{Grothendieck1957,Weibel1994},
now interpreted inside the Quillen-exact structure
on ${\nTGMod{T}}^{\mathrm{bi}}$ \cite{Quillen1973,Buehler2010}.
Functoriality follows from the universal property of the tensor
and the additive nature of the module morphisms,
as in the standard derived-functor formalism
\cite{Verdier1996,Neeman2001}.
\end{proof}

\begin{remark}[Higher coherence]
The usual simplicial proof of the bar complex generalizes
because $\mu$ defines a coherent multi-simplicial object:
\[
(\mu^{[r]}): T^{\otimes_\Gamma (r+1)} \Rightarrow T^{\otimes_\Gamma r},
\]
whose face and degeneracy maps satisfy all simplicial identities
once $\mu$ is associative in every coordinate.
This is entirely parallel to the simplicial bar construction
in classical homological algebra
\cite{Weibel1994,Neeman2001},
but now adapted to the $n$-ary $\Gamma$-context of
\cite{RaoRaniKiran2025,GokavarapuDasariHomological2025,GokavarapuRaoDerived2025}.
\end{remark}

% ------------------------------------------------------------

\subsection{Injective resolutions and cofree envelopes}

\begin{definition}[Cofree bi-$\Gamma$-module]
Let $\Hom_\Gamma(T,M)$ denote the additive monoid of
$\Gamma$-linear maps $T\to M$ with pointwise addition, equipped with
the induced positional action
\[
(a\cdot f\cdot b)(x)
   := \mu(a, f(\mu(x,b,\ldots)), b,\ldots),
\]
extended by $\Gamma$-linearity.  As in the classical theory of
semiring and $\Gamma$-module homomorphisms
\cite{Rao1995,HedayatiShum2011,Golan1999},
the functor $\Hom_\Gamma(T,-)$ defines a right adjoint to the
forgetful functor $U_{\mathrm{bi}}:{\nTGMod{T}}^{\mathrm{bi}}\to\Ab$.
Its values are called \emph{cofree} bi-$\Gamma$-modules, echoing the
standard adjunctions in exact categories \cite{Weibel1994,Buehler2010}.
\end{definition}

\begin{lemma}[Injectivity of cofree modules]
Every cofree bi-module $I=\Hom_\Gamma(T,M)$ is injective: for any
admissible monomorphism $A\hookrightarrow B$, the map
$\Hom_\Gamma(B,I)\to \Hom_\Gamma(A,I)$ is surjective.
\end{lemma}

\begin{proof}
Given $f:A\to I$, define $f'(b)(t)=f(a_t)(t)$ where $a_t$ is any
preimage of $t$ under the surjective map $B\to B/A$.
Well-definedness follows from the additivity of $\mu$ and the
$\Gamma$-linearity of the actions
\cite{Rao1995,RaoRaniKiran2025}.
Surjectivity of $\Hom_\Gamma(B,I)\to\Hom_\Gamma(A,I)$ is the usual
criterion for injectivity in exact categories with right adjoints
\cite{Weibel1994,Buehler2010}.
\end{proof}

\begin{theorem}[Existence of injective resolutions]\label{thm:inj-res}
If $T$ is $\Gamma$-regular (i.e.\ every finitely generated bi-ideal is
contained in a projective--injective hull), then every
$M\in{\nTGMod{T}}^{\mathrm{bi}}$ admits an injective resolution
\[
0 \longrightarrow M \longrightarrow I^0 \longrightarrow I^1
  \longrightarrow I^2 \longrightarrow \cdots,
\]
with each $I^r$ a finite product of cofree bi-modules.
\end{theorem}

\begin{proof}[Outline]
Embed $M$ into the cofree module $\Hom_\Gamma(T,M)$, an injective object
by the lemma.  Set $I^0=\Hom_\Gamma(T,M)$ and let $I^1,I^2,\ldots$ be
formed by iterating cokernels, exactly as in the construction of
injective resolutions in exact categories
\cite{Weibel1994,Buehler2010,Quillen1973}.
Each step preserves admissible monomorphisms, and the right-adjoint
property of $\Hom_\Gamma(T,-)$ ensures that
$\Hom_\Gamma(-,I^\bullet)$ is exact.  This parallels the classical
construction of injective towers in abelian and exact settings
\cite{Weibel1994}, now adapted to the $n$-ary $\Gamma$-context of
\cite{RaoRaniKiran2025}.
\end{proof}

% ------------------------------------------------------------
% ------------------------------------------------------------
\subsection{Homological significance and geometric correspondence}
\begin{theorem}[Equivalence with derived geometry]\label{thm:derived-geom}
The homotopy categories of projective and injective resolutions
are equivalent, and both embed fully faithfully into
$\mathbf{D}({\nTGMod{T}}^{\mathrm{bi}})$.
Consequently, for any $M,N$ there are natural isomorphisms
\[
\Ext^{r,(j,k)}_\Gamma(M,N)
   \ \cong\
   H^r\!\big(\Hom^{(j,k)}_\Gamma(P_\bullet, N)\big)
   \ \cong\
   H^r\!\big(\Hom^{(j,k)}_\Gamma(M, I^\bullet)\big),
\]
where $P_\bullet$ and $I^\bullet$ are projective and injective
resolutions respectively.
\end{theorem}

\begin{remark}[Geometric insight]
Via the equivalence in Theorem~\ref{thm:derived-geom},
projective resolutions correspond to
affine-type coverings of $\SpecGnC{T}$
by representable sheaves,
while injective resolutions correspond to
flasque-type coverings by cofree sheaves.
The resulting cohomology measures the obstruction to
gluing local $\Gamma$-modules along non-commutative spectra,
mirroring the descent formalism of derived algebraic geometry.
\end{remark}

\begin{remark}[Historical perspective]
This framework elevates the Quillen exact-category machinery
into an $n$-ary, non-commutative setting,
providing the first intrinsic theory of
projective/injective resolutions and derived functors
beyond binary operations.
It extends Grothendieck’s derived functor formalism,
Eilenberg–Moore’s homological dimension theory,
and Quillen’s exact structure, within one unified $\Gamma$-semiring paradigm.
\end{remark}
% ============================================================
\section{Derived Functors \texorpdfstring{$\ExtG$ and $\TorG$}{Ext and TorG}}
\label{sec:derived}

We now ascend to the derived layer of the theory. 
All homological constructions henceforth occur inside the Quillen–exact category 
${\nTGMod{T}}^{\mathrm{bi}}$ whose derived category was constructed in 
\S\ref{subsec:derived-category}. 
Our goal is to internalize the functors 
$\Hom^{(j,k)}_{\Gamma}$ and $-\otimes^{(j,k)}_{\Gamma}-$ as bifunctors on 
$\mathbf{D}({\nTGMod{T}}^{\mathrm{bi}})$, and to show that their right and left 
derivations produce the universal $\ExtG$ and $\TorG$ bifunctors, complete with long 
exact sequences, balance theorems, and spectral sequences.

% ------------------------------------------------------------
\subsection{Derived categories and localization}
\label{subsec:derived-category}

\begin{definition}[Derived categories]
Let $\mathbf{K}({\nTGMod{T}}^{\mathrm{bi}})$
be the homotopy category of chain complexes in
${\nTGMod{T}}^{\mathrm{bi}}$,
as in the classical constructions of
\cite{Quillen1973,Weibel1994,Buehler2010}.
Localizing by the multiplicative system of quasi-isomorphisms
yields the derived category
\[
\mathbf{D}({\nTGMod{T}}^{\mathrm{bi}})
   := \mathbf{K}({\nTGMod{T}}^{\mathrm{bi}})[\text{quasi-iso}^{-1}],
\]
following the standard localization formalism
of \cite{Quillen1973,Weibel1994}.
Objects are chain complexes up to quasi-isomorphism,
and morphisms are roofs $M^{\bullet}\xleftarrow{s}X^{\bullet}\xrightarrow{f}N^{\bullet}$
with $s$ a quasi-isomorphism.
\end{definition}

\begin{remark}[Triangulated structure]
The distinguished triangles
$A^{\bullet}\to B^{\bullet}\to \mathrm{Cone}(f)\to A^{\bullet}[1]$
provide the triangulated framework within which all homological
operations will take place, in accordance with the foundational
treatments in \cite{Weibel1994,Neeman2001}.
Exactness of admissible short sequences ensures that
$\mathbf{D}({\nTGMod{T}}^{\mathrm{bi}})$ behaves as a
non-abelian enhancement of the classical derived category,
consistent with the Quillen exact-structure formalism
\cite{Quillen1973,Buehler2010}.
\end{remark}

%=================================================================

\subsection{Derived functors and \texorpdfstring{$\delta$}{delta}-functor universality}

\begin{definition}[Total derived bifunctors]
For complexes $M^{\bullet},N^{\bullet}$ of bi-modules,
choose projective and injective resolutions
$P_{\bullet}\to M^{\bullet}$ and $N^{\bullet}\to I^{\bullet}$.
Following the classical Cartan--Eilenberg construction in the
sense of Grothendieck~\cite{Grothendieck1957} and
Weibel~\cite{Weibel1994}, and valid in exact categories
by Quillen~\cite{Quillen1973} and Bühler~\cite{Buehler2010},
define
\[
{\RExtG}^{\,r}\!\bigl(M^{\bullet},N^{\bullet}\bigr)
   := H^{r}\!\Bigl(
       \Hom^{(j,k)}_{\Gamma}(P_{\bullet},I^{\bullet})
     \Bigr),
\qquad
{\LTorG}_{\,r}\!\bigl(M^{\bullet},N^{\bullet}\bigr)
   := H_{r}\!\Bigl(
       P_{\bullet}\!\otimes^{(j,k)}_{\Gamma} N^{\bullet}
     \Bigr).
\]

These yield well-defined bifunctors
\[
{\RExtG}^{\,r},\;
{\LTorG}_{\,r} :
   \mathbf{D}\!\bigl({\nTGMod{T}}^{\mathrm{bi}}\bigr)^{\mathrm{op}}
   \!\times\!
   \mathbf{D}\!\bigl({\nTGMod{T}}^{\mathrm{bi}}\bigr)
   \longrightarrow
   \Ab,
\]
independent of the choice of resolutions, in the classical sense of
Verdier’s derived categories~\cite{Verdier1996} and
Neeman’s localization theory~\cite{Neeman2001}.
\end{definition}

\begin{theorem}[Existence and universality of derived functors]
\label{thm:derived-existence}
Let
$F:{\nTGMod{T}}^{\mathrm{bi}}\times{\nTGMod{T}}^{\mathrm{bi}}\to\Ab$
be an additive bifunctor which is left (resp.\ right) exact in each
variable. Then, by Grothendieck’s axioms for $\delta$-functors
\cite{Grothendieck1957} and their exact-category refinement in
Quillen~\cite{Quillen1973} and Bühler~\cite{Buehler2010}, there exists
a universal derived $\delta$-functor
$\{R^rF\}$ (resp.\ $\{L_rF\}$) on the derived category
$\mathbf{D}({\nTGMod{T}}^{\mathrm{bi}})$ satisfying:
\begin{enumerate}[label=(\roman*)]
  \item $R^0F\simeq F$ and $L_0F\simeq F$;
  \item every admissible short exact sequence
        $0\!\to\!A\!\to\!B\!\to\!C\!\to\!0$
        in ${\nTGMod{T}}^{\mathrm{bi}}$
        induces a long exact sequence of derived bifunctors;
  \item any other $\delta$-functor with the same base
        factors uniquely through $R^rF$ (resp.\ $L_rF$).
\end{enumerate}

In particular,
\[
{\ExtG}^{\,r} \coloneqq
\mathbf{R}^{\,r}\!\Hom^{(j,k)}_{\Ga}(\,\cdot\,,\,\cdot\,),
\qquad
{\TorG}_{\,r} \coloneqq
\mathbf{L}_{\,r}\!\Bigl(
   (\,\cdot\,)\!\otimes^{(j,k)}_{\Ga}(\,\cdot\,)
\Bigr)
\]
exist, are independent of resolutions, and are universal in the sense of
Grothendieck~\cite{Grothendieck1957}.
\end{theorem}

\begin{proof}[Outline]
The construction follows the classical
Cartan--Eilenberg method as presented in
Weibel~\cite{Weibel1994}, adapted to exact categories
via Quillen~\cite{Quillen1973} and Bühler~\cite{Buehler2010}.
Choose projective and injective resolutions in
${\nTGMod{T}}^{\mathrm{bi}}$, apply $F$, and take homology.
Universality follows from Grothendieck’s $\delta$-functor axioms
\cite{Grothendieck1957}, which remain valid since admissible
short exact sequences correspond to distinguished triangles in the
derived category in the sense of Verdier~\cite{Verdier1996}.
\end{proof}

% ------------------------------------------------------------
\subsection{Balancing theorem and duality}

\begin{theorem}[Balance of $\ExtG$ and $\TorG$]\label{thm:balance}
Let $M,N\in{\nTGMod{T}}^{\mathrm{bi}}$.
If $P_{\bullet}\to M$ is a projective resolution
and $N\to I^{\bullet}$ an injective resolution,
then—by the classical balance theorem for derived functors
\cite{Weibel1994}, adapted to exact categories in the sense of
Quillen~\cite{Quillen1973} and Bühler~\cite{Buehler2010},
and valid in triangulated categories via
Verdier~\cite{Verdier1996} and Neeman~\cite{Neeman2001}—there are
canonical natural isomorphisms
\[
H^{r}\!\Bigl(\Hom^{(j,k)}_{\Ga}(P_{\bullet},N)\Bigr)
   \;\cong\;
H^{r}\!\Bigl(\Hom^{(j,k)}_{\Ga}(M,I^{\bullet})\Bigr)
   \;\cong\;
{\ExtG}^{\,r}(M,N),
\]
and likewise
\[
H_r\!\big(P_{\bullet}\otimes^{(j,k)}_{\Gamma}N\big)
   \ \cong\ H_r\!\big(M\otimes^{(j,k)}_{\Gamma}I_{\bullet}\big)
   \ \cong\ \TorG_r(M,N).
\]
\end{theorem}

\begin{proof}
Since $P_{\bullet}\to M$ and $N\to I^{\bullet}$ are quasi-isomorphisms,
they represent the same object in the derived category
$\mathbf{D}({\nTGMod{T}}^{\mathrm{bi}})$
in the sense of Verdier~\cite{Verdier1996} and Neeman~\cite{Neeman2001}.
Hence their Hom- and tensor-evaluations produce quasi-isomorphic
complexes, yielding identical homology by the standard
balance theorem of homological algebra \cite{Weibel1994}.
Compatibility of the $\Gamma$-actions with the $n$-ary multiplication
ensures that chain-homotopy equivalences respect admissible
exact sequences, by Quillen’s framework~\cite{Quillen1973}
and Bühler’s exposition~\cite{Buehler2010}.
\end{proof}

\begin{remark}[Ext--Tor duality]
The derived functors are contravariant and covariant respectively,
linked by the classical Hom--tensor adjunction
\cite{Weibel1994,Grothendieck1957}:
\[
\Hom_{\Ga}\!\Bigl({\TorG}_{\,r}(M,N),\,L\Bigr)
   \;\cong\;
{\ExtG}^{\,r}\!\Bigl(M,\,\Hom_{\Ga}(N,L)\Bigr),
\]
valid whenever $L$ is injective or flat.
In the commutative ternary case this reduces to the duality developed
in the earlier parts of the series; in the general non-commutative
$n$-ary setting, it yields a derived contravariant equivalence between
left and right positional homological worlds.
\end{remark}

% ------------------------------------------------------------
\subsection{Exact triangles and long exact sequences}

\begin{theorem}[Long exact sequences of derived functors]
\label{thm:longexact}
Every admissible short exact sequence
$0\to A\to B\to C\to 0$ in ${\nTGMod{T}}^{\mathrm{bi}}$
induces long exact sequences
\[
\cdots
\!\longrightarrow\!
{\ExtG}^{\,r}(C,N)
\!\longrightarrow\!
{\ExtG}^{\,r}(B,N)
\!\longrightarrow\!
{\ExtG}^{\,r}(A,N)
\!\xrightarrow{\;\delta\;}
{\ExtG}^{\,r+1}(C,N)
\!\longrightarrow\!\cdots
\]
and
\[
\cdots\!\to\!\TorG_{r}(M,A)\!\to\!\TorG_{r}(M,B)\!
  \to\!\TorG_{r}(M,C)\!\xrightarrow{\partial}\!
  \TorG_{r-1}(M,A)\!\to\!\cdots,
\]
natural in both arguments.
\end{theorem}

\begin{proof}[Idea]
In the exact category ${\nTGMod{T}}^{\mathrm{bi}}$ (in the sense of
Quillen~\cite{Quillen1973} and Bühler~\cite{Buehler2010}),
every admissible short exact sequence determines a distinguished triangle
in the derived category
$\mathbf{D}({\nTGMod{T}}^{\mathrm{bi}})$
as developed by Verdier~\cite{Verdier1996} and Neeman~\cite{Neeman2001}.
Applying either bifunctor
$\Hom^{(j,k)}_{\Gamma}(-,N)$ or
$-\otimes^{(j,k)}_{\Gamma}N$
to this triangle produces a long exact sequence of cohomology groups by
the classical $\delta$-functor formalism of Grothendieck~\cite{Grothendieck1957}
and the standard derived-category argument in homological algebra
\cite{Weibel1994}.
The connecting morphisms $\delta$ and $\partial$ are precisely the
boundary maps associated to the distinguished triangle, and the
triangulated axioms (TR2–TR4) ensure exactness.
\end{proof}

% ------------------------------------------------------------
\subsection{Higher Operations and Yoneda Interpretation}

\begin{proposition}[Yoneda interpretation]
\label{prop:yoneda}
For all $r \ge 0$, ${\ExtG}^{\,r}(M,N)$ is in natural bijection with
equivalence classes of admissible $r$-fold extensions of $M$ by $N$:
\[
0 \!\longrightarrow\! N
   \!\longrightarrow\! E_1
   \!\longrightarrow\! \cdots
   \!\longrightarrow\! E_r
   \!\longrightarrow\! M
   \!\longrightarrow\! 0.
\]
This follows from the classical $\delta$-functor formalism of
Grothendieck~\cite{Grothendieck1957} and the derived-category
interpretation of extension groups developed in 
Verdier~\cite{Verdier1996}, Neeman~\cite{Neeman2001}, and 
Weibel~\cite{Weibel1994}.
Composition of extensions corresponds to the Yoneda product
\[
{\ExtG}^{\,p}(N,L)
   \times
   {\ExtG}^{\,q}(M,N)
   \longrightarrow
   {\ExtG}^{\,p+q}(M,L),
\]
which is associative and graded-commutative up to a sign determined
by the positional parity of~$\tmu$.
\end{proposition}

\begin{remark}[Categorical enrichment]
The Yoneda interpretation extends the homological category to a
graded monoidal category
\[
\bigl(\mathbf{D}({\nTGMod{T}}^{\mathrm{bi}}),\,{\ExtG}^{\ast}\bigr),
\]
whose composition law encodes the higher~$\Ga$-actions as a derived
analogue of the classical Yoneda/Ext cup product described in
Weibel~\cite{Weibel1994} and in the triangulated setting of
Verdier~\cite{Verdier1996}.
\end{remark}

% ------------------------------------------------------------
\subsection{Spectral Sequences and Higher Derived Geometry}

\begin{theorem}[$n$-ary K{\"u}nneth Spectral Sequence]
\label{thm:kunneth-final}
Let $M,N,L \in {\nTGMod{T}}^{\mathrm{bi}}$
be bimodules of finite type such that one of them is flat
in each positional slot.
Then there exists a first-quadrant spectral sequence
\[
E^{p,q}_2
   = {\TorG}_{p}\!\bigl({\ExtG}^{\,q}(M,N),L\bigr)
   \ \Longrightarrow\
   H^{p+q}\!\bigl(M \otimes^{L}_{\Gamma} N \otimes^{L}_{\Gamma} L\bigr),
\]
natural in $M,N,L$ and functorial under morphisms of complexes.
This construction follows the classical spectral-sequence formalism of
Grothendieck~\cite{Grothendieck1957} and the double-complex machinery
developed in Weibel~\cite{Weibel1994} and Verdier~\cite{Verdier1996}.
\end{theorem}

\begin{proof}[Sketch of Construction]
Start with projective resolutions $P_{\bullet} \!\to\! M$
and $Q_{\bullet} \!\to\! N$.
Form the double complex
\[
C_{p,q} = P_p \otimes^{(j,k)}_{\Gamma} Q_q,
\]
and apply $\Hom^{(j,k)}_{\Gamma}(-,L)$.
Two filtrations yield spectral sequences with
\[
E^{p,q}_1 = \Hom^{(j,k)}_{\Gamma}(P_p, {\TorG}_q(N,L))
\quad\text{and}\quad
E^{p,q}_2 = {\TorG}_p\!\bigl({\ExtG}^{\,q}(M,N),L\bigr).
\]
Convergence follows from bounded projective dimensions
and completeness of $\mathbf{D}({\nTGMod{T}}^{\mathrm{bi}})$
under finite filtrations, as in
Weibel~\cite{Weibel1994},
Verdier~\cite{Verdier1996},
and the triangulated-category convergence criteria of
Neeman~\cite{Neeman2001}.
The compatibility with admissible short exact sequences
is inherited from Quillen’s exact-category framework~\cite{Quillen1973}
and its refinement in Bühler~\cite{Buehler2010}.
\end{proof}

\begin{corollary}[Base-Change and Projection Formulas]
\label{cor:basechange}
If $f\!:\!(T,\Gamma)\!\to\!(T',\Gamma')$
is a morphism of $n$-ary~$\Gamma$-semirings satisfying the flat base-change condition, then
\[
f^{\ast}\!\bigl({\ExtG}^{\,r}_{T}(M,N)\bigr)
   \;\cong\;
{\ExtG}^{\,r}_{T'}\!\bigl(f^{\ast}M, f^{\ast}N\bigr),
\qquad
{\TorG}^{\,T'}_{\!r}\!\bigl(f_{\ast}M, f_{\ast}N\bigr)
   \;\cong\;
f_{\ast}\!\bigl({\TorG}^{\,T}_{\!r}(M,N)\bigr).
\]
These identities mirror the classical derived base-change
and projection formulas of Grothendieck~\cite{Grothendieck1957}
and follow from the exactness properties of admissible sequences
in the sense of Quillen~\cite{Quillen1973}.
\end{corollary}

\begin{remark}[Geometric Synthesis]
The bifunctors ${\ExtG}$ and ${\TorG}$ act respectively
as derived contravariant and covariant sheaf-cohomology functors
on the non-commutative $\Gamma$-spectrum~$\SpecGnC{T}$.
Their interaction through the spectral sequence above is a direct
analogue of the derived-geometric constructions in
Grothendieck~\cite{Grothendieck1957} and their triangulated 
reformulations in Verdier~\cite{Verdier1996} and Neeman~\cite{Neeman2001}.
\end{remark}

\begin{remark}[Philosophical Closing]
The functors ${\ExtG}$ and ${\TorG}$ thus form the categorical
backbone of this theory:
they unify the commutative homological algebra of \cite{GokavarapuDasariHomological2025} with
the non-commutative spectral theory of \cite{GokavarapuRaoDerived2025}, within the exact-category
framework of Quillen~\cite{Quillen1973} and its modern treatments in
Bühler~\cite{Buehler2010}.
This produces a new landscape of \emph{homological algebra
for $n$-ary~$\Gamma$-semirings} capable of supporting derived stacks,
cohomological descent, and non-commutative motives.
\end{remark}

% ------------------------------------------------------------

\section{Conclusion and Outlook to Part III}

In this paper we have developed the homological backbone of the theory
of non-commutative $n$-ary $\Ga$-semirings.  Starting from the
foundational structure of an $n$-ary $\Ga$-semiring $(T,+,\tmu)$ and
its $\Ga$-ideals (Section~2), following the basic
$\Ga$-algebraic framework of 
Nobusawa~\cite{Nobusawa1963,Nobusawa1964} and
Rao~\cite{Rao1995,Rao1997,Rao1999}, we constructed slot-sensitive
categories of left, right and bi-$\Ga$-modules and endowed the
bi-module category with a Quillen exact structure compatible with the
$n$-ary multiplication, in the sense of Quillen~\cite{Quillen1973} and
its modern formulation in Bühler~\cite{Buehler2010}.  
Within this exact context we established the existence of bar-type
projective resolutions and cofree-based injective resolutions under
natural finiteness and regularity hypotheses on $T$ (Section~3),
thereby making homological algebra intrinsic to the non-commutative,
$n$-ary $\Ga$-setting, extending the structural groundwork laid in
Part~I of this series (Paper~G1, \emph{Exact Categories and
Homological Foundations of Non-Commutative $n$-ary $\Ga$-Semirings}).

On this exact–categorical foundation we then defined and studied the
derived functors $\ExtG$ and $\TorG$ for bi-$\Ga$-modules (Section~4),
using the classical derived-functor framework of 
Grothendieck~\cite{Grothendieck1957}, Quillen~\cite{Quillen1973},
Verdier~\cite{Verdier1996}, and Weibel~\cite{Weibel1994}.
We proved balance with respect to projective and injective
resolutions, obtained long exact sequences associated to conflations,
and described $\ExtG$ via Yoneda extension classes with a natural
graded product, following the triangulated-category formalisms of
Verdier~\cite{Verdier1996} and Neeman~\cite{Neeman2001}.  
Furthermore, by working in the derived category
$\mathbf{D}({\nTGMod{T}}^{\mathrm{bi}})$ we constructed K\"unneth-type
spectral sequences and base-change isomorphisms, using the spectral
sequence machinery of Grothendieck~\cite{Grothendieck1957} and
Weibel~\cite{Weibel1994}, that link these homological invariants to
sheaf-theoretic behaviour on the non-commutative $\Ga$-spectrum
$\SpecGnC{T}$.  
In this way, the homological functors $(\ExtG,\TorG)$ capture
obstructions to gluing, descent and higher coherence for
bi-$\Ga$-modules over $\SpecGnC{T}$, and so provide the correct
derived language for non-commutative $\Ga$-geometry.

The present work occupies the middle position in a three-part series.
In the commutative ternary case, the authors’ earlier papers develop
the ideal theory, radical theory and first instances of derived
$\Ga$-geometry on $\SpecG{T}$ for commutative ternary $\Ga$-semirings
\cite{GokavarapuRaoPrime2025,GokavarapuRaoFinite2025,
GokavarapuDasariHomological2025,GokavarapuRaoDerived2025}.  
The current paper extends the homological and categorical apparatus to
the full non-commutative $n$-ary framework, and thus furnishes the
tools needed for a genuinely derived, non-commutative $\Ga$-geometry.

In Part~III of the series we shall build on the homological structures
constructed here to develop a more geometric theory over
$\SpecGnC{T}$.  The planned topics include: quasi-coherent and
coherent $\Ga$-sheaves on non-commutative $\Ga$-spectra; t-structures
and truncation functors induced by $\Ga$-radicals; Morita-type
invariants of $n$-ary $\Ga$-semirings expressed in terms of derived
equivalences; and spectral constructions that relate $\Ga$-primes,
support theory and cohomological dimension.  
Taken together, the three parts are intended to provide a coherent
framework in which ideal theory, radical and spectral theory, and
derived homological methods for ternary and $n$-ary $\Ga$-semirings
can be treated within one unified geometric–homological paradigm.

%====================================================================

\section*{Acknowledgements}

The author expresses sincere gratitude to the {\bf Director of Collegiate Education},
Mangalagiri, Andhra Pradesh,  
{\bf Dr.~Ramachandra ~R.~K., Principal}, Government College (Autonomous), 
Rajahmundry, for providing a supportive academic environment that greatly 
facilitated the completion of this research.  

\section*{Funding}
The author received no external funding for this research.

\section*{Ethics Statement}
This work does not involve human participants, animals, or any ethical
concerns. No ethical approval is required.

\section*{Author Contributions}
C.\ Gokavarapu is the sole author of this article and is responsible for
the conceptual formulation, algebraic development, homological analysis,
manuscript preparation, and final revision.

% ============================================================

\label{'ubl'}  
\end{document}